\documentclass[12pt]{article}
\usepackage[utf8]{inputenc}
\usepackage{amsmath,amsthm,amssymb}
\usepackage{thmtools}
\usepackage{physics}
\usepackage{mathtools}
\usepackage{mathrsfs}

\newtheorem{thm}{Theorem}[section]
\newtheorem{lem}[thm]{Lemma}
\newtheorem{cor}[thm]{Corollary}

\newtheorem{conj}[thm]{Conjecture}

\theoremstyle{definition}			                						
\newtheorem{mydef}[thm]{Definition}
\newtheorem{rem}[thm]{Remark}

\title{Properties of the cone of polynomials of fixed degree that preserve nonnegative matrices}
\author{Jared Brannan \and Benjamin J.~Clark \and Garrett Kepler}
\date{August 2024}

\begin{document}

\maketitle

\begin{abstract}
As was detailed by Loewy and London in [Linear and Multilinear Algebra 6 (1978/79), no.~1, 83--90], the cone of polynomials that preserve the nonnegativity of matrices may play an important role in the solution to the nonnegative inverse eigenvalue problem. In this paper, we start by showing the cone generated by polynomials of degree greater than or equal to $2n$ that preserve nonnegative matrices of order $n$ is non-polyhedral. Next, a question posed by Loewy in [Linear Algebra and its Applications, 676(2023), 267--276], about how negative the center term can be in a degree $2n$ polynomial is answered. We extend this to show that a polynomial that preserves nonnegative matrices of order $n$ can have it's the largest term, in absolute value, be arbitrarily negative with the remaining coefficients being one. We conclude, by exploring properties of the measure of the cone when restricted to the unit sphere and by proving some initial bounds of that volume.
\end{abstract}

\section{Introduction}

Motivated by the nonnegative inverse eigenvalue problem (NIEP), Loewy and London in \cite{loewy1978} posed the problem of characterizing the polynomials that preserve nonnegative matrices of a fixed order. Finding a characterization has been shown to provide additional necessary conditions for the NIEP. See \cite{johnson2018} for a survey on the NIEP and \cite{loewy2023} for more information on how characterizing the polynomials that preserve nonnegative matrices gives necessary conditions for the NIEP.

We denote $\mathcal{P}_n$ to be the set of polynomials that preserve nonnegative matrices of order $n$. Clark and Paparella in \cite{clark2022} gave general results for the coefficients of polynomials in $\mathcal{P}_n$ and in \cite{clark2021} gave a characterization for $\mathcal{P}_2$. Bharali and Holtz in \cite{bharali2008} considered a larger class of entire functions that preserve nonnegative matrices of a fixed order. Loewy in \cite{loewy2023} gave a proof of a conjecture of Clark and Paparella that $\mathcal{P}_{n+1}$ is a strict subset of $\mathcal{P}_n$ and showed that if we restrict the degree of polynomials in $\mathcal{P}_n$ to some $m \in \mathbb{Z}^+$, denoted $\mathcal{P}_{n,m}$, then $\mathcal{P}_{2,m}$ is non-polyhedral for $m \geq 4$. We extend this result to show that $\mathcal{P}_{n,m}$ is non-polyhedral for all $n \geq 2$ and $m \geq 2n$.

Next, we answer \cite[Question 2.1]{loewy2023} by showing that the optimal bound is $2$ for any given $n$. With this, we extend the result to give a new set of polynomials in $\mathcal{P}_n$ where the center term can be made arbitrarily negative given enough positive terms on the edge.

We conclude by considering the volume the cone $\mathcal{P}_{n,m}$ takes when restricted to the unit sphere. We conjecture that as $m$ goes to infinity, the volume of the cone goes to $0$. Some inequalities are given for how the volume changes as we vary $m$ and $n$ for $\mathcal{P}_{n,m}$.

\section{Background and notation}

The set $\mathsf{M}_n$ is used for $n$ by $n$ real square matrices. The set restrictions $\mathsf{M}_n^{\geq 0}$ ($\mathsf{M}_n^{+})$ will be used to denote the $n$ by $n$ entry-wise nonnegative (positive) matrices.

The set $\mathcal{P}_n = \{ p \in \mathbb{R}[x] : p(A) \geq 0, \forall A \in \mathsf{M}_n^{\geq 0}\}$ is the set of polynomials that preserve the nonnegativity of matrices of a fixed order and is the main object of study. It is shown in \cite[Lemma 6]{clark2021} that $\mathcal{P}_n$ is equivalent to $\{ p \in \mathbb{R}[x] : p(A) \geq 0, \forall A \in \mathsf{M}_n^{+}\}$, that is we need only consider positive matrices. When we restrict the degree of polynomials in $\mathcal{P}_n$ we will use the notation of \cite{loewy2023}, which is $\mathcal{P}_{n,m} = \{ p \in \mathcal{P}_n : \text{deg}(p) \leq m \}$.

For the convex analysis dealing with the cone behaviors of $\mathcal{P}_n$ and $\mathcal{P}_{n,m}$ we will be following the notation of \cite{rockafellar1997convex}. Let $S$ be a set of values where scalar multiplication and addition is defined, then 
\[
\text{cone}(S) = \left\{ \sum_{i=1}^k \alpha_i x_i : x_i \in S, \alpha_i \in \mathbb{R}^{\geq 0}, k \in \mathbb{N} \right\}
\] 
denotes the conical hull of the elements of $S$. When working with the elements of $S$ we will call them the conical generating elements.

Let $p \in \mathbb{R}[x]$, then denote 
\[
\mathcal{Z}_{i,j}(p) := \{ A \in \mathsf{M}^+_n : (p(A))_{i,j} = 0 \}.
\]
Define $\mathcal{Z}(p)$ to be the union of $\mathcal{Z}_{i,j}(p)$ for all $i,j \in \{1, 2, \dots, n\}$.

\begin{lem} \label{lem:positive_stochastic}
    Let $p \in \mathbb{R}[t]$, then $p \in \mathcal{P}_n$ if and only if $p(\rho A) \geq 0$ for all $A \in M_n$ where $A$ is a positive stochastic matrix and $\rho \in \mathbb{R}^+$.
\end{lem}

\begin{proof}
    This follows directly from \cite[Lemma 1]{clark2021}, \cite[Lemma 6]{clark2021}, and noting that all positive matrices can be turned into a row stochastic matrix by a diagonal similarity.
\end{proof}

We are following the definition and construction of Hausdorff measure from Evans and Gariepy, \cite[Definition 2.1]{evans2015}.

\begin{mydef}
    Let $A \subseteq \mathbb{R}^n$, then $\mathcal{H}^s$ denotes the $s$-dimensional Hausdorff measure on $\mathbb{R}^n$ and $\mathcal{H}^s(A)$ denotes the $s$-dimensional Hausdorff measure of the set $A$.
\end{mydef}

Denote $\overline{B}_k(0,1)$ to be the closed $k$ dimensional unit sphere centered at $0$.

\section{Cone of polynomials is non-polyhedral for sufficient degree}

\begin{mydef}
    Let $p,q \in \mathbb{R}[x]$ and $t \in [0,1]$, then define 
    \[
    g_{p,q,t}(x) := (1-t)p(x) + tq(x)
    \]
    to be the polynomial made from the convex combination of $p$ and $q$.
\end{mydef}

\begin{lem} \label{lem:extremal_face}
    Let $X = \{x, x^2, \dots, x^{2n}\}$ and $P = \{p_i\}_{i=1}^m \subset \mathcal{P}_{n,2n}$ where 
    \[
    p_i(x) = \sum_{\substack{k=0 \\ k \not = n}}^{2n} a_{i,k} x^k - x^n.
    \]
    For $t \in (0,1)$, $p \in P$, and $q \in X$, the ray through $g_{p,q,t}$ is not an exposed ray in $\mathcal{P}_{n,2n}$.
\end{lem}

\begin{proof}
Loewy in \cite[Theorem 2.1]{loewy2023} showed that $P$ is nonempty. Let $p \in P$ and $q \in X$, then $\mathcal{Z}(g_{p,q,t}) = \emptyset$ for all $t \in (0,1)$ following $\mathcal{Z}(q) = \emptyset$. For a specific $t \in (0,1)$ we can find an $\epsilon > 0$ such that
\[
h(A) = g_{p,q,t}(A) - \epsilon \sum_{\substack{k=0 \\ k \not = n}}^{2n} a_{i,k} A^k \geq 0
\]
for all $A \in \mathsf{M}_{n}^{\geq 0}$ following $x^k$ is continuous over matrices for all $k \in \{0, \dots, 2n\}$. Thus the ray through $g_{p,q,t}$ is not an exposed ray following $g_{p,q,t}$ can be expressed as a convex combination of $p,q,h$.
\end{proof}

\begin{thm} \label{thm:P_2n_not_polyhedral}
    Let $n \in \mathbb{Z}^+$, then $\mathcal{P}_{n,2n}$ is not a polyhedral cone.
\end{thm}

\begin{proof}
    Assume for contradiction that $\mathcal{P}_{n,2n}$ is finitely generated, then $\mathcal{P}_{n,2n}$ is a superset of the nonnegative orthant, generated by $\{1,x,x^2, \dots, x^{2n}\}$. Clark and Paparella, in \cite[Corollary 4.2 and Theorem 4.11]{clark2022}, showed that the first and last $n$ terms of the polynomial must be nonnegative. However, Loewy showed in \cite[Corollary 2.2]{loewy2023} that the $x^n$ term can be negative. Without loss of generality, we can assume that the $n$th coefficient is equal to $-1$. Thus the remaining finite conical generating elements of $\mathcal{P}_{n,2n}$ are $P = \{p_i\}_{i=1}^m$ where
    \[
    p_i(x) = \sum_{\substack{k=0 \\ k \not = n}}^{2n} a_{i,k} x^k - x^n
    \]
    and with $a_{i,k} \geq 0$.

    Pick $p \in P$ and $q \in X = \{x,x^2, \dots, x^{2n}\}$ such that the convex combination of them is on the boundary of $\text{cone}(P \cup Q)$. By Lemma \ref{lem:extremal_face} for some $t\in (0,1)$ we have that $g_{p,q,t}(x)$ is not an exposed ray. Let $u(x)$ be the normal polynomial with respect to the face through $\{g_{p,q,t}: t\in [0,1]\}$. Perturb the coefficients of $g_{p,q,t}$ in the direction of $u$ so that it is exposed ray, call this new polynomial $h$. By construction $h$ is not in $\text{cone}(P \cup Q)$. Thus the ray through $h$ is a new exposed ray of $\mathcal{P}_{n,2n}$ which contradicts $P$ and $\{1,x,x^2, \dots, x^{2n}\}$ generating $\mathcal{P}_{n,2n}$.
\end{proof}

\begin{cor}
    Let $n,m \in \mathbb{Z}^+$ such that $m > 2n$, then $\mathcal{P}_{n,m}$ is not a polyhedral cone.
\end{cor}

\begin{proof}
    If a convex cone is polyhedral, then every face must be polyhedral. Loewy in \cite[Lemma 3.1]{loewy2023} showed that $\mathcal{P}_{n,2n}$ is a face of $\mathcal{P}_{n,m}$ and by Theorem \ref{thm:P_2n_not_polyhedral} that face is not polyhedral.
\end{proof}

\section{Properties of the negative coefficients}

\begin{thm} \label{thm:t=2}
    The largest $t \in \mathbb{R}^+$ such that 
    \[
    p_t(x) = \sum_{k=0}^{n-1} x^{k+s} - tx^m + \sum_{k=0}^{n-1} x^{2m + k -s} \in \mathcal{P}_n
    \]
    for all $m,s \in \mathbb{Z}^{\geq 0}$ such that $m \geq n$ and $s \leq m-n$ is $t = 2$.
\end{thm}

\begin{proof}
    By Lemma \ref{lem:positive_stochastic} for $p_t \in \mathcal{P}_n$ it suffices to consider $p_t(\rho A) \geq 0$ for all stochastic $A \in \mathsf{M}_n^+$ and $\rho \in \mathbb{R}^+$. We associate with the entries of $A$ with the weights of a directed fully connected graph $G$ on $n$ nodes. Then, the $(i,j)$th entry of $(\rho A)^m$ can be viewed as the sum of the product of the weights of walks of length $m$ in $G$. We will show that each of the walks in $(2(\rho A)^m)_{ij}$ can be bounded by walks from $\sum_{k=0}^{n-1} (\rho A)^{k+s}$ or from $\sum_{k=0}^{n-1} (\rho A)^{2m + k-s}$.

    Let $P_{i,j}$ denote a walk of length $m-\ell$ from vertex $i$ to vertex $j$ on $G$ where $\ell \in \mathbb{Z}^{\geq 0}$ such that $s \leq m - \ell \leq s+n-1$. The product of the weights of this walk contributes to the sum of $(A^{m-\ell})_{i,j}$ which is part of the sum from $k=0$ to $n-1$. To make $P_{i,j}$ a walk of length $m$, we need to introduce cycles that account to a total length of $\ell$. Let $w$ be the sum of the product of the weights of all those cycles. Now we can traverse these cycles twice to make walks of length $m+\ell$. These double cycles will have a sum product of cycle weights of $w^2$. This gives our walk polynomial associated with the entry $i,j$ of $p_t(\rho A)$ as
    \[
    \rho^{m - \ell} - tw \rho^{m} + w^2 \rho^{m + \ell}
    = \rho^{m - \ell} \left( 1 - t(w \rho^{\ell}) + (w \rho^{\ell})^2\right).
    \]
    Since $\rho^{m-\ell}>0$, for the above polynomial to be nonnegative we need 
    \[
    1 - t(w \rho^{\ell}) + (w \rho^{\ell})^2 \geq 0
    \]
    which is true for all $\rho,w \in \mathbb{R}$ when $t=2$. Notice that when $t > 2$ there exists a $\rho$ and $w$ such that the above inequality is false. Thus, $2$ is the maximum for $t$.
\end{proof}

\begin{rem}
    When $s=0$ and $m=n$, Theorem \ref{thm:t=2} answers \cite[Question 2.1]{loewy2023}. The bound for $a$ in Loewy's $p_a$ is 2 for all $n \in \mathbb{Z}^+$. 
\end{rem}

\begin{cor} \label{cor:middle_negative}
    Let $\alpha \in \mathbb{R}^+$ and $m = n \lceil \alpha / 2 \rceil$, then 
    \[
    \sum_{k=0}^{m - 1} x^k - \alpha x^{m} + \sum_{k=m + 1}^{2m} x^k \in \mathcal{P}_n.
    \]
\end{cor}

\begin{proof}
    Let $\alpha \in \mathbb{R}^+$, $m = n \lceil \alpha / 2 \rceil$, and
    \[
    p(x) = \sum_{k=0}^{m - 1} x^k - \alpha x^{m} + \sum_{k=m + 1}^{2m} x^k,
    \]
    then we can separate $p$ into 
    \[
    p_s(x) = \sum_{k=sn}^{n(s + 1) - 1} x^k - 2x^{m} + \sum_{k=m + 1 + sn}^{m + n(s+1)} x^k
    \]
    where $s \in \{0, \dots, \lceil \alpha / 2 \rceil\}$. By Theorem \ref{thm:t=2} each of the $p_s \in \mathcal{P}_n$, thus $p \in \mathcal{P}_n$. 
\end{proof}

\begin{conj}
    The largest $t \in \mathbb{R}^+$ such that 
    \[
    p_t(x) = \sum_{k=0}^{n-1} x^{k} - tx^m + \sum_{k=0}^{n-1} x^{s + k} \in \mathcal{P}_n
    \]
    for all $m,s \in \mathbb{Z}^{\geq 0}$ such that $s > m \geq n$ is a continuous function of $m$ and $s$ bounded between 1 and 2.
\end{conj}

\section{Measure of polynomials restricted to the unit sphere}

In this section, we will consider the volume that $\mathcal{P}_{n,k}$ takes within the unit sphere. We are treating $\mathcal{P}_{n,k}$ as a subset of a $k+1$ dimensional vector space with respect to the coefficients of the polynomials.

\begin{lem} \label{lem:hyperplanes}
    Let $C_1$ and $C_2$ be convex compact subsets of $\mathbb{R}^n$ with nonempty interior such that $C_2 \subset C_1$, then $\mathcal{H}^n(C_2) < \mathcal{H}^n(C_1)$. 
\end{lem}

\begin{proof}
    Let $C_1$ and $C_2$ be convex, compact subsets of $\mathbb{R}^n$ with non-empty interior such that $C_2 \subset C_1$ and let $x \in C_1 \setminus C_2$. By the Separation Theorem \cite[p. 344]{cheney} there exist two open halfspaces $H_1$ and $H_2$ separated by the hyperplane $S_x$ where $S_x$ is a supporting hyperplane of $C_2$, $x \in H_1$ and $C_2 \subset H_2$, and where $\text{dist}(S_x, x) = \epsilon > 0$. Let $\widehat{S}_x$ be the family of supporting hyperplanes of the convex hull of $C_2 \cup \{x\}$ containing $x$. Let $C$ be the convex cone bounded by $\widehat{S}_x$ and $S_x$. Note, by construction $C \subseteq C_1$.

    Note, we can associate with each ray from $x$ through $C_2$ an angle element $\theta$ and the point $y_\theta \in S_x$. Let $R_{\theta} (t) = tx + (1-t) y_{\theta}$ where $0 < t < 1$. By construction, 
    \[
    \mathcal{H}^1\left( \{R_{\theta}(t) : 0 < t < 1\}\right) \geq > 0.
    \]
    Let $\widehat{C}$ be the set of all such $\theta$s. Writing $\theta = (\theta_1, \dots, \theta_n)$, we have that 
    \[
    \mathcal{H}^1 \left( \{\sigma_j : \exists~ \theta \in \widehat{C} \text{ with } \theta = (\theta_1, \dots, \theta_j = \sigma_j, \dots, \theta_n)\}\right) \not = 0
    \]
    for all $j \in \{1, 2, \dots, n\}$.
    Otherwise $\mathcal{H}^n(C_2) = 0$ contradicting that $C_2$ has a nonempty interior. Then by the coarea formula \cite[Theorem 3.10]{evans2015} we have 
    \[
    \mathcal{H}^n(C) = \int_{\widehat{C}} \mathcal{H}^1 \left( \{ R_{\theta}(t) : 0 < t < 1\} \right) ~d\theta > 0.
    \]
    Therefore with $C_2 \cap \text{int} C = \emptyset$ we have
    \[
    \mathcal{H}^n(C_2) < \mathcal{H}^n(C_2 \cup C) \leq \mathcal{H}^n(C_1). \qedhere
    \]
\end{proof}

\begin{mydef}
    Define $\pi: \mathcal{P}_{n,k+1} \rightarrow \mathcal{P}_{n,k}$ by projecting off the $k+1$ term of the polynomial.
\end{mydef}

Note that the projection operator is continuous.

\begin{lem} \label{lem:isometric}
    Let $n,k \in \mathbb{Z}^+$ such that $k \geq 2n$, then $\mathcal{P}_{n,k} \cap \overline{B}_{k+1}$ is a convex compact subset of $\mathbb{R}^{k+1}$  with nonempty interior.
\end{lem}

\begin{proof}
    By \cite{clark2022}, $\mathcal{P}_n$ and $\mathcal{P}_{n,k}$ are convex and, as was shown in \cite{loewy2023}, $\mathcal{P}_{n,k}$ is closed. The nonempty interior follows from $\mathcal{P}_{n,k}$ always containing the nonnegative orthant of $\mathbb{R}^{k+1}$. Finally following $\overline{B}_{k+1}$ forms a compact subset and $\mathcal{P}_{n,k}$ is closed we get that $\mathcal{P}_{n,k} \cap \overline{B}_{k+1}$ is compact.
\end{proof}

\begin{lem} \label{lem:subset_measure}
    Let $n,k \in \mathbb{Z}^+$ such that $k \geq 2n$, then $\mathcal{P}_{n,k} \subset \pi(\mathcal{P}_{n,k+1})$.
\end{lem}

\begin{proof}
    By Theorem \ref{thm:t=2} there exists polynomials in $\pi(\mathcal{P}_{n,k+1})$ such that their $k - n$ term is negative. By \cite[Theorem 4.11]{clark2022} $\mathcal{P}_{n,k}$ does not have polynomials with their $k - n$ term negative. The result follows, noting that $\mathcal{P}_{n,k} \subset \mathcal{P}_{n,k+1}$.
\end{proof}

\begin{thm}
    Let $n,k \in \mathbb{Z}^+$ such that $k \geq 2n$, then 
    \[
    \mathcal{H}^{k+1} \left( \pi ( \mathcal{P}_{n,k+1} ) \cap \overline{B}_{k+1}(0,1) \right) > \mathcal{H}^{k+1} \left( \mathcal{P}_{n,k} \cap \overline{B}_{k+1} (0,1) \right)
    \]
\end{thm}

\begin{proof}
    Follows from Lemma \ref{lem:hyperplanes}, Lemma \ref{lem:subset_measure}, and Lemma \ref{lem:isometric}.
\end{proof}

\begin{thm}
    Let $n,k \in \mathbb{Z}^+$ such that $k \geq 2n$, then 
    \[
    \mathcal{H}^{k+1} \left( \mathcal{P}_{n+1,k} \cap \overline{B}_{k+1}(0,1) \right) < \mathcal{H}^{k+1} \left( \mathcal{P}_{n,k} \cap \overline{B}_{k+1} (0,1) \right).
    \]
\end{thm}

\begin{proof}
    The result follows from Lemma \ref{lem:isometric}, \cite[Corollary 2.1]{loewy2023} giving that $\mathcal{P}_{n+1,k} \subset \mathcal{P}_{n,k}$, and by Lemma \ref{lem:hyperplanes}.
\end{proof}

\begin{conj}
    Let $n,k \in \mathbb{Z}^+$ such that $k \geq 2n$, then 
    \[
    \mathcal{H}^{k+2} \left(  \mathcal{P}_{n,k+1} \cap \overline{B}_{k+2}(0,1) \right) < \mathcal{H}^{k+1} \left( \mathcal{P}_{n,k} \cap \overline{B}_{k+1} (0,1) \right)
    \]
\end{conj}

\begin{conj}
    Let $n,k \in \mathbb{Z}^+$ such that $k \geq 2n$, then
    \[
    \lim_{k \rightarrow \infty} \frac{\mathcal{H}^{k+1} \left( \mathcal{P}_{n,k} \cap \overline{B}_{k+1} (0,1) \right)}{\mathcal{H}^{k+1}\left(\overline{B}_{k+1}(0,1)\right)} = 0.
    \]
\end{conj}

\begin{rem}
    We are not sure on the behavior of $\mathcal{H} \left( \mathcal{P}_{n} \cap  \overline{B}_{\infty}(0,1) \right)$. In infinite dimensions, we lose the compactness of the sphere. While $\mathcal{P}_n$ does not contain any infinite dimensional polynomials, it contains sequences of polynomials that go to infinite dimensions. To study this further it may be easier to consider the set of entire functions instead of polynomials as was done in \cite{bharali2008}.
\end{rem}

\section{Concluding remarks}

At the end of the previous section we have two conjectures. One potential avenue for solving these are with further study of $\mathcal{P}_{1,k}$. The characterization of $\mathcal{P}_1$ is known as the Pólya–Szegö theorem, see \cite[Proposition 2]{powers2000}, which asserts that $p \in \mathcal{P}_1$ if and only if 
\[
p(x) = f_1(x)^2 + f_2(x)^2 + x \left(g_1(x)^2 + g_2(x)^2 \right)
\]
where $f_1,f_2,g_1,g_2 \in \mathbb{R}[t]$. Following that $\mathcal{P}_n \subset \mathcal{P}_1$ this means that all polynomials in $\mathcal{P}_n$ can be written as sums of squares. 

Sums of squares tie in closely with quadratic forms. Determining an expression for the curve that extends out of the nonnegative orthant $\mathcal{P}_{n,2n}$ would be beneficial to characterizing $\mathcal{P}_n$. We offer the following conjecture for a potential avenue of determining that curve.

\begin{conj} \label{conj:characterizing_f}
    There exists a finite $k \in \mathbb{Z}^+$ where $f_k(\textbf{v}) \geq 0$ for $f_k \in \mathbb{R}^{2n}[x]$ implies the polynomial whose coefficients are $\textbf{v}$ is in $\mathcal{P}_{n,2n}$.
\end{conj}

Another way to state the above conjecture is that there are a finite number of quadratic forms that describe the boundary of $\mathcal{P}_{n,2n}$.

\section*{Acknowledgements}

We would like to thank Patrick C. Gambill for his ideas in section 4 and overall good vibes. We would also like to thank Kevin R. Vixie for his ideas in section 5.

\bibliographystyle{abbrv}
\bibliography{cone_properties}

\end{document}